\documentclass[reqno]{amsart}
\usepackage{txfonts}
\usepackage{amsfonts}
\usepackage{amssymb}
\usepackage[bookmarksopen,colorlinks,citecolor=blue,
linkcolor=black,pdfstartview=FitH]{hyperref}
\usepackage[mathscr]{eucal}
\usepackage{amsmath}
\usepackage{enumitem}

\newtheorem{theorem}{Theorem}[section]
\newtheorem{corollary}[theorem]{Corollary}
\newtheorem{lemma}[theorem]{Lemma}
\newtheorem{proposition}[theorem]{Proposition}

\theoremstyle{definition}

\theoremstyle{remark}
\newtheorem{remark}[theorem]{Remark}

\numberwithin{equation}{section}

\allowdisplaybreaks[4]

\begin{document}

\title{A class of generalized positive linear maps on matrix algebras}
\author{Xin Li}
\address{Department of Mathematics, East China Normal University, Shanghai 200241, China}
\email{tsoete@163.com}
\author{Wei Wu}
\address{Department of Mathematics, East China Normal University, Shanghai 200241, China}
\email{wwu@math.ecnu.edu.cn}
\thanks{Corresponding author: Wei Wu}
\thanks{The research was supported in part by Shanghai Leading Academic Discipline Project (Project No.
B407), and National Natural Science Foundation of China (Grant No. 11171109).}


\begin{abstract}
We construct a class of positive linear maps on matrix algebras. We find conditions when these maps are atomic, decomposable and completely positive. We obtain a large class of atomic positive linear maps. As applications in quantum information theory, we discuss the structural physical approximation and optimality of entanglement witness associated with these maps.
\end{abstract}
\maketitle

\section{Introduction}\label{se:1}

Positive linear maps on $C^{*}-$algebras, particularly those of finite dimensions, have been becoming more important by their connection with quantum information theory. A linear map on a $C^{*}$-algebra is called {\it positive} if it sends the cone of positive elements into itself. Little is known about the global structure of positive linear maps, even in the low dimensional matrix algebras.
Let $M_{n}$ be the $C^{*}$-algebra of all $n\times n$ matrices over the complex field, and let $\mathcal{P}_{k}(M_n)$ (respectively, $\mathcal{P}^{k}(M_n)$) be the convex cone of all $k$-positive (respectively, $k$-copositive)
linear maps on $M_n$. One of the basic problems about the structures of the
positive cone $\mathcal{P}_{1}(M_n)$ is whether the set $\mathcal{P}_{1}(M_n)$ can be decomposed as the algebraic
sum of some simpler classes in $\mathcal{P}_{1}(M_n)$ \cite{Tomi}.
When $n=2$, it is well known \cite{Woronowicz} that
every positive linear map can be written as a sum of a completely positive linear map and a completely copositive linear map, that is, the maps in $\mathcal{P}_{1}(M_2)$ are decomposable. But this is not the case for higher dimensional matrix algebras. On $M_3$, Choi gave an extremal positive linear map which is indecomposable \cite{Choi7512}. Tanahashi and Tomiyama in \cite{Tomi} introduced the concept of atomic positive linear map which has a stronger indecomposability, and they showed that Choi's map is atomic. There are only a few examples of indecomposable positive linear maps in the literature, much less the atomic ones. Most known examples of indecomposable positive linear maps and atomic positive linear maps can be found in \cite{Chrus07, Chrus08, Chrus12, Ha03, Kye1992, Kye2012} and references therein. In quantum information theory, indecomposable positive linear maps can be used to detect entangled states whose partial transposes are positive and atomic positive linear maps can be used to detect states with the `weakest' entanglement \cite{Chrus08}. Positive linear maps also play an important role in the study of operator system theory \cite{PTT, LW}, etc.

In this paper, we give a generalization of linear maps defined in \cite{Ha03}. Let $S_{n}$ be the symmetric group consisting of all bijections (permutations) from the set $\{1, 2,\ldots, n\}$ onto
itself. For positive real numbers $a,c_{1},c_{2},\ldots,c_{n}$ and each $\sigma\in S_{n}$, we define a linear map $\Theta^{(n,\sigma)}[a;c_{1},c_{2},\ldots,c_{n}]$ from $M_{n}$ to $M_{n}$ by
\begin{align*}
\Theta^{(n,\sigma)}[a;c_{1},c_{2},\ldots,c_{n}](X)=
\Delta^{(n,\sigma)}[a;c_{1},c_{2},\ldots,c_{n}](X)-X,
\end{align*}
where
\begin{align*}
&\Delta^{(n,\sigma)}[a;c_1,c_{2}, \ldots ,c_n ](X)\\
=&\left(
      \begin{array}{cccc}
        a x_{11}+c_{1}x_{\sigma(1),\sigma(1)} & 0 & \cdots & 0  \\
        0 & a x_{22}+c_{2} x_{\sigma(2),\sigma(2)} & \cdots & 0 \\
      \vdots & \vdots &  \ddots & \vdots  \\
        0 & 0 & \cdots & a x_{nn}+c_{n} x_{\sigma(n),\sigma(n)} \\
      \end{array}\right),
\end{align*}
for each $X=(x_{ij})\in M_{n}$. Let $\phi:M_{n}\mapsto M_{n}$ be a linear map. If it has the form
\begin{align}\label{de:D-type}
\phi: (x_{ij})\mapsto \text{diag} (f_{1},\ldots,f_{n})-(x_{ij})\qquad \text{with}\quad (f_{1},\ldots,f_{n})=(x_{11},\ldots,x_{nn})D
\end{align}
where $D=(d_{ij})$ is an $n\times n$ nonnegative matrix, that is, all $d_{ij}\geq 0$, then $\phi$ is called a {\it $D$-type linear map} \cite{HouLi12}. In (\ref{de:D-type}), if we let
\[D=aI_{n}+\sum_{i=1}^{n}c_{i}E_{\sigma(i)i}\]
where $I_{n}$ and $\{E_{ij}\}_{i,j=1}^{n}$ are the identity matrix and the canonical matrix units of $M_{n}$, respectively, we can see that $\Theta^{(n,\sigma)}[a;c_{1},c_{2},\ldots,c_{n}]$ has the form in
(\ref{de:D-type}) and so it is a $D$-type linear map. Throughout this paper, if there is no confusion, $\Theta^{(n,\sigma)}[a;c_{1},c_{2},\ldots,c_{n}]$ and $\Delta^{(n,\sigma)}[a;c_1,c_{2},\ldots ,c_n]$ will often be abbreviated to $\Theta^{(n,\sigma)}$ and $\Delta^{(n,\sigma)}$, respectively.

For each $k\in \{1,2,\ldots,n\}$, we define $\tau_{k}^{n}\in S_{n}$ by
\begin{align}\label{eq:translation}
\tau_{k}^{n}(i)\equiv i+k\pmod n,
\end{align}
for $i=1,2,\ldots,n$. The linear map $\Theta^{(3,\tau_{2}^{3})}[a;c_{1},c_{2},c_{3}]$ was studied in \cite{Kye1992}. In \cite{Ha03}, Ha defined the map $\Theta^{(n,\tau_{n-1}^{n})}[a;c_{1},c_{2},\ldots,c_{n}]$ which is a generalization of $\Theta^{(3,\tau_{2}^{3})}[a;c_{1},c_{2},c_{3}]$ and gave a sufficient condition for the map $\Theta^{(n,\tau_{n-1}^{n})}[a;c_{1},c_{2},\ldots,c_{n}]$ being atomic. In \cite{Hou11}, Qi and Hou defined the map  $\Theta^{(n,\tau_{k}^{n})}[n-1;1,1,\ldots,1]$ and discussed when $\Theta^{(n,\tau_{k}^{n})}$ is positive and indecomposable.
In \cite{Hou12}, Qi and Hou studied the optimality, decomposability and structural physical approximation of $\Theta^{(n,\tau_{k}^{n})}[n-1;1,1,\ldots,1]$ for $k\neq n$. For each $\sigma\in S_{n}$ and $c\geq 0$, the positivity of $\Theta^{(n,\sigma)}[n-c;c,c,\ldots,c]$ was discussed in \cite{HouLi12}. For $\sigma^{2}=id_{n}$ where $id_{n}$ is the identity of $S_{n}$, the decomposability of $\Theta^{(n,\sigma)}[n-1;1,1,\ldots,1]$ was also discussed in \cite{HouLi12}. In \cite{Ha12}, Ha discussed the optimality of the entanglement witness associated with $\boldsymbol{\mathrm{T}}\circ \Theta^{(n,\tau_{k}^{n})}[n-1;1,1,\ldots,1]$ for $k\neq n$ and $\frac{n}{2}$ (when $n$ ($n\geq 3$) is even), where $\boldsymbol{\mathrm{T}}$ denotes the transpose map.

The paper is organized as follows. In Section \ref{se:p2p} we give conditions when the map $\Theta^{(n,\sigma)}[a;c_{1},c_{2},\ldots,c_{n}]$ is positive and discuss the equivalence between $2$-positivity and completely positivity. In Section \ref{se:atomic} we give conditions when $\Theta^{(n,\sigma)}[a;c_{1},c_{2},\ldots,c_{n}]$ is atomic and decomposable. We give conditions in Section \ref{se:spaopt} when the structural physical approximation of $\Theta^{(n,\sigma)}[a;c_{1},c_{2},\ldots,c_{n}]$ is separable and the entanglement witness associated with $\boldsymbol{\mathrm{T}}\circ \Theta^{(n,\sigma)}[a;c_{1},c_{2},\ldots,c_{n}]$ is optimal.

Throughout this paper, a matrix $A$ is positive means that $A$ is positive semi-definite and is denoted by $A \geq 0$. For every vector in $\mathbb{C}^{n}$, we consider it as an $n\times 1$ matrix, that is, a column vector. If $x$ is a vector or a matrix, then $x^{t}$ and $x^{*}$ denote the transpose and conjugate transpose of $x$, respectively. Let $\{e_{i}\}_{i=1}^{n}$ and $\{E_{ij}\}_{i,j=1}^{n}$ denote the canonical orthonormal basis of $\mathbb{C}^{n}$ and the matrix units of $M_{n}$, respectively. Let $\langle\cdot,\cdot\rangle$ be the usual inner product on $\mathbb{C}^{n}$ and $(n,k)$ denote the greatest common divisor of $n$ and $k$. For $m,n\in \mathbb{N}$, if $m$ divides $n$ we write $m|n$, and if $m$ does not divide $n$ we write $m\nmid n$. Let $\boldsymbol{\mathrm{T}}$ denote the transpose map on $M_{n}$ and $id_{n}$ denote the identity of $S_{n}$.

The authors are grateful to the referee for careful reading of the manuscript and several helpful comments.

\section{Positivity and 2-positivity}\label{se:p2p}

In this section, we give conditions when $\Theta^{(n,\sigma)}$ is positive and then discuss the equivalence between 2-positivity and completely positivity.

\begin{lemma}{\rm{(\cite{Ha03})}}\label{le:symf}
Let $a>0$. For symmetric function $$F(x_1,\ldots,x_n)=\sum_{m=0}^n\left[a^{m-1}(a-m)\sum_{1\leq i_1<\ldots<i_{n-m}\leq n}x_{i_1}\cdots x_{i_{n-m}}\right],$$
where $x_1,\ldots,x_n$ are positive real numbers, we have that $F(x_1,\ldots,x_n)\geq 0 $ if and only if $\sum_{i=1}^n (a+x_{i})^{-1}\leq 1$.
\end{lemma}

\begin{lemma}{\rm{(\cite{Ha03})}}\label{le:2ies}
For $x_i\geq 0$, $i=1,2,\ldots,n$, $x\geq 0$ and real number $a$, we have the following:
\begin{align}
\sum_{1\leq i_{1}<\cdots<i_{n-m}\leq n}x_{i_1}\cdots x_{i_{n-m}}\geq \frac{n!}{(n-m)!m!}(x_{1}\cdots x_{n})^{\frac{n-m}{n}}, \quad 0\leq m< n;\label{2ies-1}
\end{align}
\begin{align}
(x^{\frac{1}{n}}+a)^{n-1}(x^{\frac{1}{n}}+a-n)=\sum_{m=0}^{n} a^{m-1}
(a-m)\frac{n!}{(n-m)!m!}x^{\frac{n-m}{n}}.\label{2ies-2}
\end{align}
\end{lemma}

A permutation $\sigma\in S_{n}$ is called a {\it cycle of length $k$} ($k=1,2,\ldots,n$) if for $k$ distinct points $\{i_{1},i_{2},\ldots,i_{k}\}\subseteq\{1,2,\ldots,n\}$, we have that $\sigma(i_{j})=i_{j+1}$ ($j=1,2,\ldots k-1$),
$\sigma(i_{k})=i_{1}$ and $\sigma(i)=i$ for all $i\in\{1,2,\ldots,n\}\backslash \{i_{1},i_{2},\ldots,i_{k}\}$. In the following, denote $l(\sigma)$ be the length of a cycle $\sigma$.
It is well known (for example \cite{GTM163}) that each $\sigma\in S_{n}$ has a unique disjoint cycle decomposition $\sigma=\sigma_{1}\sigma_{2}\cdots\sigma_{r}$, where each $\sigma_{i}$ ($i=1,2,\ldots, r$) is a cycle. In the following, for each $\sigma\in S_{n}$ with the unique disjoint cycle decomposition $\sigma=\sigma_{1}\sigma_{2}\cdots\sigma_{r}$, we denote the maximal and the minimal length of $\sigma_{i}$ ($i=1,2,\ldots, r$) by $l_{\max}(\sigma)$ and $l_{\min}(\sigma)$ respectively, that is,
$$l_{\max}(\sigma)=\max\{l(\sigma_{1}),l(\sigma_{2}),\ldots,l(\sigma_{r})\}, $$
and $$l_{\min}(\sigma)=\min\{l(\sigma_{1}),l(\sigma_{2}),\ldots,l(\sigma_{r})\}.$$

Suppose that $k\in\{1,2,\ldots,n\}$. If $k|n$, it is not hard to see that $\tau_{k}^{n}$ (defined in (\ref{eq:translation})) can be decomposed into $k$ disjoint cycles and each cycle has length $\frac{n}{k}$. For $k \nmid n$, if $(n,k)=r$, then $r|n$ and each $i\in \{1,2,\ldots,n\}$ can be written as $i=u+rv$, where $1\leq u \leq r$ and $0\leq v \leq \frac{n}{r}-1$. Just as in \cite{HaYu12}, define $\sigma\in S_{n}$ by
$$\sigma(i)=\sigma(u+rv)\equiv u+kv \pmod n.$$
It is not hard to see that $\tau_{r}^{n}=\sigma^{-1}\tau_{k}^{n}\sigma $, that is, $\tau_{r}^{n}$ and $\tau_{k}^{n}$ are conjugate in $S_{n}$. Hence $\tau_{r}^{n}$ and $\tau_{k}^{n}$ have the same number of cycles of each type \cite{GTM163}, that is, $\tau_{k}^{n}$ can be decomposed into $r$ disjoint cycles and each cycle has length $\frac{n}{r}$. So for each $k\in \{1,2,\ldots,n\}$, we have that $l_{\min}(\tau_{k}^{n})=l_{\max}(\tau_{k}^{n})=\frac{n}{(n,k)}$. It is not hard to see that if $k\neq n$ and $\frac{n}{2}$ (when $n$ ($n\geq 3$) is even), then $l_{\min}(\tau_{k}^{n})=l_{\max}(\tau_{k}^{n})\geq 3$. Hence we have the following lemma.

\begin{lemma}\label{le:length}
Suppose that $k\in\{1,2,\ldots,n\}$ . Let $\tau_{k}^{n}$ be the permutation defined in (\ref{eq:translation}). Then we have that $l_{\min}(\tau_{k}^{n})=l_{\max}(\tau_{k}^{n})=\frac{n}{(n,k)}$ and the following:
\begin{enumerate}[label=\textup{(\roman*)}, ref=\textup{\roman*}]
\item\label{le:lengthn}if $k=n$ , then $l_{\min}(\tau_{k}^{n})=l_{\max}(\tau_{k}^{n})=1$;
\item\label{le:lengthn/2}if $k=\frac{n}{2}$ when $n$ is even, then $l_{\min}(\tau_{k}^{n})=l_{\max}(\tau_{k}^{n})=2$;
\item\label{le:lengthno} if $k\neq n$ and $\frac{n}{2}$ (when $n$ ($n\geq 3$) is even), then $l_{\min}(\tau_{k}^{n})=l_{\max}(\tau_{k}^{n})\geq 3$.
\end{enumerate}
\end{lemma}

\begin{lemma}\label{le:mainie}
Let $a, c_1,c_{2} \ldots, c_n$ be positive real numbers. For each $\sigma\in S_{n}$, if
\begin{align}
a\geq \max\{n-1, n-(c_1c_{2}\cdots c_n)^{\frac{1}{n}}\},
\end{align}
we have the following inequality
\begin{align}
\frac{\alpha_1}{a\alpha_1+c_1\alpha_{\sigma(1)}}
+\frac{\alpha_2}{a\alpha_2+c_2\alpha_{\sigma(2)}}
+\frac{\alpha_3}{a\alpha_3+c_3\alpha_{\sigma(3)}}+\cdots
+\frac{\alpha_n}{a\alpha_n+c_n\alpha_{\sigma(n)}}\leq 1 \label{ie0}
\end{align}
for any positive real numbers $\alpha_1,\alpha_2,\ldots,\alpha_n$. If $\sigma$ is a cycle of length $n$, then the converse is also held.
\end{lemma}

\begin{proof}
Suppose that $a\geq \max\{n-1, n-(c_1c_{2}\cdots c_n)^{\frac{1}{n}} \}$. Let $x_i= c_i \frac{\alpha_{\sigma(i)}}{\alpha_i}$ for $i\in \{1,2,\ldots,n\}$, then $x_{1}x_{2}\cdots x_{n}=c_{1}c_{2}\cdots c_{n}$. For $F(x_1,\ldots,x_n)$ in Lemma \ref{le:symf}, we have
\begin{align}
F(x_1,\ldots,x_n)&=\sum_{m=0}^n\left[a^{m-1}(a-m)\sum_{1\leq i_1<\ldots<i_{n-m}\leq n}x_{i_1}\cdots x_{i_{n-m}}\right]\nonumber\\
&=\sum_{m=0}^{n-1}\left[a^{m-1}(a-m)\sum_{1\leq i_1<\ldots<i_{n-m}\leq n}x_{i_1}\cdots x_{i_{n-m}}\right]+a^{n-1}(a-n)\nonumber\\
&\geq  \sum_{m=0}^{n-1} a^{m-1}(a-m) \frac{n!}{(n-m)!m!}(x_{1} x_{2} \cdots x_{n})^{\frac{n-m}{n}} + a^{n-1}(a-n) \label{le:mainie-1}\\
&=\sum_{m=0}^n  a^{m-1}(a-m) \frac{n!}{(n-m)!m!} (x_{1} x_{2} \cdots x_{n})^{\frac{n-m}{n}}\nonumber\\
&=((c_{1}c_{2} \cdots c_n)^{\frac{1}{n}}+a)^{n-1}(a+(c_{1}c_{2}\cdots c_n)^{\frac{1}{n}}-n)\label{le:mainie-2}\\
&\geq 0.\label{le:mainie-3}
\end{align}
Since $a\geq n-1$, we have that $a-m\geq 0$ for $m=0,1,\ldots,n-1$ and (\ref{le:mainie-1}) is obtained by (\ref{2ies-1}) in Lemma \ref{le:2ies}. From (\ref{2ies-2}) of Lemma \ref{le:2ies}, we have (\ref{le:mainie-2}). Since $a\geq n-(c_1c_{2}\cdots c_n)^{\frac{1}{n}}$, we have (\ref{le:mainie-3}). Hence by Lemma \ref{le:symf} we get the desired inequality
$$\sum_{i=1}^n \frac{\alpha_i}{a\alpha_i+c_i \alpha_{\sigma(i)}}\leq 1.$$

Conversely, suppose that (\ref{ie0}) holds for any positive real numbers $\alpha_1,\alpha_2,\ldots,\alpha_n$ and $\sigma$ is a cycle of length $n$. It is not hard to see that $\{\sigma^{1}(1),\sigma^{2}(1),\ldots,\sigma^{n}(1)
\}=\{1,2,\ldots,n\}$.

First, we show that $a\geq n-1$. For any $\lambda>0$, we choose
$$\alpha_{i}=\lambda^{-s}\quad\text{if}\quad i=\sigma^{s}(1),$$
where $s\in \{1,2,\ldots, n\}$. Note that $1=\sigma^{n}(1)$. So if $i=\sigma^{s}(1)$ with $s\in\{1,2,\ldots,n-1\}$, then for $i=2,3,\ldots,n$ we have that $\sigma(i)=\sigma^{s+1}(1)$. Hence we have
$$\frac{\alpha_{\sigma(i)}}{\alpha_i}=\frac{\lambda^{-(s+1)}}{\lambda^{-s}}
=\frac{1}{\lambda},~i=2,3,\ldots,n ,$$
and
$$ \frac{\alpha_{\sigma(1)}}{\alpha_1}=\frac{\lambda^{-1}}{\lambda^{-n}}=\lambda^{n-1}.$$
Now from (\ref{ie0}) we have
\begin{align}
\sum_{i=1}^n \frac{\alpha_i}{a\alpha_i+c_i\alpha_{\sigma(i)}}&
=\frac{1}{a+c_{1}\frac{\alpha_{\sigma(1)}}{\alpha_1}}
+\frac{1}{a+c_{2}\frac{\alpha_{\sigma(2)}}{\alpha_2}}+\cdots
+\frac{1}{a+c_{n}\frac{\alpha_{\sigma(n)}}{\alpha_n}}\nonumber\\
&=\frac{1}{a+c_{1}\lambda^{n-1}}+\frac{1}{a+\frac{c_{2}}{\lambda}}
+\cdots+\frac{1}{a+\frac{c_{n}}{\lambda}}\leq 1\label{ie1}.
\end{align}
Take $\lambda\to +\infty$, then we have that $\frac{n-1}{a}\leq 1$ by (\ref{ie1}). So we obtain that $a\geq n-1$.

Next, we show that $a\geq n-(c_{1}c_{2}\cdots c_{n})^{\frac{1}{n}}$. Let $d=(c_{1}c_{2} \cdots c_{n})^{\frac{1}{n}}$. For each $i\in \{1,2,\ldots,n\}$, if $i=\sigma^{k}(1)$ for some $k\in\{1,2,\ldots,n\}$, we let
\begin{align}\label{id2}
\alpha_{i}=\alpha_{\sigma^{k}(1)}=\frac{\alpha_{1} d^{k}}{c_{\sigma^{0}(1)}c_{\sigma(1)}\cdots c_{\sigma^{k-1}(1)}},
\end{align}
where $\sigma^{0}(1)=1$.

For $i,k\in\{1,2,\ldots,n\}$, if $i=\sigma^{k}(1)$, then from (\ref{id2}) we have
\begin{align}\label{eq:sigmaid}
c_i\frac{\alpha_{\sigma(i)}}{\alpha_i}&=
c_{\sigma^{k}(1)}\frac{\alpha_{\sigma(\sigma^{k}(1))}}
{\alpha_{\sigma^{k}(1)}}
=c_{\sigma^{k}(1)}\frac{\alpha_{\sigma^{k+1}(1)}}
{\alpha_{\sigma^{k}(1)}}\nonumber\\
&=c_{\sigma^{k}(1)}\frac{\alpha_{1} d^{k+1}}{c_{\sigma^{0}(1)}c_{\sigma(1)}\cdots c_{\sigma^{k}(1)}}\frac{c_{\sigma^{0}(1)}c_{\sigma(1)}\cdots c_{\sigma^{k-1}(1)}}{\alpha_{1} d^{k}}\\
&=d\nonumber.
\end{align}
Hence from (\ref{eq:sigmaid}) and (\ref{ie0}) we have
\begin{align*}
\sum_{i=1}^n \frac{\alpha_i}{a\alpha_i+c_i\alpha_{\sigma(i)}}=
\sum_{i=1}^{n}\frac{1}{a+c_{i}\frac{\alpha_{\sigma(i)}}{\alpha_i}}
=\frac{n}{a+d}\leq 1.
\end{align*}
So we get $a\geq n-d=n-(c_{1}c_{2}\cdots c_n)^{\frac{1}{n}}$.

From discussions above, we have that $a\geq \max\{n-1, n-(c_{1}c_{2}\cdots c_n)^{\frac{1}{n}} \}$.
\end{proof}

\begin{lemma}{\rm{(\cite{Tomi})}}\label{le:inv}
Let $A$ be a positive invertible operator on a Hilbert space, and $\xi_{0}$ the unit vector associated with a one dimensional projection P. Then $A \geq P$ if and only if $\langle A^{-1}\xi_0, \xi_0\rangle\leq 1$.
\end{lemma}

\begin{theorem}\label{th:positive}
Let $a, c_1,c_{2} \ldots, c_n$ be positive real numbers. For each $\sigma\in S_{n}$, if $a\geq \max\{n-1, n-(c_1c_2\cdots c_n)^{\frac{1}{n}} \}$, then $\Theta^{(n,\sigma)}[a;c_1,c_{2} \ldots, c_n]:M_n \mapsto M_n$ is positive. Moreover, if $\sigma$ is a cycle of length $n$, then the converse is also held.
\end{theorem}

\begin{proof}
$\Theta^{(n,\sigma)}$ is positive if and only if $\Theta^{(n,\sigma)}(P)\geq 0$ for every one dimensional projection $P$, which means $\Delta^{(n,\sigma)}(P)\geq P$. Let $\xi_{0}=(x_1,\ldots,x_n)^{t}$ be the unit vector associated with $P$, that is, $P=\xi_{0}\xi_{0}^{*}$. Without loss of generality, we can assume that $x_i \neq 0$ for $i=1,\ldots,n$. Then the matrix $\Delta^{(n,\sigma)}(P)$ has the form
$$ \left(
      \begin{array}{cccc}
        a |x_1|^{2}+c_1|x_{\sigma(1)}|^{2} & 0 & \cdots & 0  \\
        0 & a|x_2|^{2}+c_2|x_{\sigma(2)}|^{2} & \cdots & 0 \\
      \vdots & \vdots &  \ddots & \vdots  \\
        0 & 0 & \cdots & a|x_n|^{2}+c_{n}|x_{\sigma(n)}|^2  \\

      \end{array}
    \right).$$
Hence $A=\Delta^{(n,\sigma)} (P)$ is invertible and positive. From Lemma \ref{le:inv} we can see that $\Theta^{(n,\sigma)}$ is positive if and only if
\begin{align}\label{ieinv}
\langle A^{-1}\xi_{0},\xi_{0}\rangle &=\frac{|x_1|^{2}}{a|x_1|^{2}+c_1|x_{\sigma(1)}|^{2}}
  +\frac{|x_2|^{2}}{a|x_2|^{2}+c_2|x_{\sigma(2)}|^{2}}+\cdots+
   \frac{|x_n|^{2}}{a|x_n|^{2}+c_n|x_{\sigma(n)}|^{2}}\nonumber\\
&\leq 1.
\end{align}
By Lemma \ref{le:mainie} and (\ref{ieinv}), the proof is completed.
\end{proof}

\begin{remark}
For $\Theta^{(n,\sigma)}[a;c_1,c_{2} \ldots, c_n]$, suppose that $c\geq 0$, $a=n-c$ and $c_{1}=c_{2}=\cdots=c_{n}=c$. The map $\Theta^{(n,\sigma)}[n-c;c,c,\ldots,c]$ is discussed in Proposition 6.2 of \cite{HouLi12}. For any $\sigma\in S_{n}$ which is not necessarily a cycle of length $n$, Hou, Li et al. showed that $\Theta^{(n,\sigma)}[n-c;c,c,\ldots,c]$ is positive if and only if $c\leq \frac{n}{l_{max}(\sigma)}$. In this case, we can see that there exists $\sigma\in S_{n}$ such that the positivity of $\Theta^{(n,\sigma)}[a;c_{1},c_{2},\ldots,c_{n}]$ cannot imply ``$a\geq \max\{n-1, n-(c_1c_2\cdots c_n)^{\frac{1}{n}} \}$". So for general $\sigma\in S_{n}$, it is interesting to find a necessary and sufficient condition for the positivity of $\Theta^{(n,\sigma)}[a;c_{1},c_{2},\ldots,c_{n}]$.
\end{remark}

Suppose that $\phi:M_{n}\mapsto M_{n}$ is a linear map. For any positive integer $k$, let $M_{k}(M_{n})$ denote the block matrix algebra of order $k$ over $M_n$. Equivalently, $M_{k}(M_{n})$ is often written as $M_{k}\otimes M_{n}$. Then we can define two linear maps $\phi_{k}$ and $\phi^{k}$ on $M_{k}\otimes M_{n}$ by
\begin{align*}
\phi_{k}\left((a_{ij})_{1\leq i,j\leq k}\right)= \left(\phi(a_{ij})\right)_{1\leq i,j\leq k}
\end{align*}
and
\begin{align*}
\phi^{k}\left((a_{ij})_{1\leq i,j\leq k}\right)= \left(\phi(a_{ij}^{t})\right)_{1\leq i,j\leq k},
\end{align*}
where $a_{ij} \in M_{n}$ for $i,j=1,2,\ldots,k$. We say that $\phi$ is {\it $k$-positive} (or {\it $k$-copositive}) if $\phi_{k}$ (or $\phi^{k}$) is positive. If $\phi_{k}$ (or $\phi^{k}$) is positive for all $k=1,2,\ldots$, then $\phi$ is said to be {\it completely positive} (or {\it completely copositive}).

The {\it Choi matrix} of a linear map $\psi:M_{n}\mapsto M_{n}$ is defined by
$$C_{\psi}=\sum_{i,j=1}^{n} E_{ij}\otimes\psi(E_{ij}) \in M_n\otimes M_n.$$
It is well known \cite{Choi7510} that $\psi$ is completely positive if and only if $C_{\psi}$ is positive. It is not hard to see that $\psi$ is completely copositive if and only if $\boldsymbol{\mathrm{T}}\circ\psi$ is completely positive.

\begin{theorem}\label{th:cp2p}
Let $a,c_{1},c_{2},\ldots,c_{n}$ be positive real numbers. For $\sigma\in S_{n}$, if $l_{\min}(\sigma)\geq 2$, then the following are equivalent:
\begin{enumerate}[label=\textup{(\roman*)}, ref=\textup{\roman*}]
\item\label{cp2p1} the linear map $\Theta^{(n,\sigma)}[a; c_{1},c_{2},\ldots,c_{n}]$ is completely positive;
\item\label{cp2p2} the linear map $\Theta^{(n,\sigma)}[a; c_{1},c_{2},\ldots,c_{n}]$ is 2-positive;
\item\label{cp2p3} $a\geq n$.
\end{enumerate}
\end{theorem}

\begin{proof}
$(\ref{cp2p1}) \Rightarrow (\ref{cp2p2})$ is clear by definition.
For $(\ref{cp2p2}) \Rightarrow (\ref{cp2p3})$, we assume that $\Theta^{(n,\sigma)}$ is 2-positive. Let $\xi=(x_{1},x_{2},\ldots,x_{n},y_{1},y_{2},\ldots,y_{n})^{t}\in \mathbb{C}^{2n}$ with $\|\xi\|=1$. Let $x=(x_{1},x_{2},\ldots,x_{n})^{t}$, $y=(y_{1},y_{2},\ldots,y_{n})^{t}\in \mathbb{C}^{n}$. Then
$$P=\xi\xi^{*}=\left(
\begin{array}{cc}
 xx^{*} & xy^{*} \\
 yx^{*} & yy^{*} \\
 \end{array}
\right)\in M_{2n}.$$
It is clear that $P$ is a projection, and so we have $\Theta^{(n,\sigma)}_{2}(P)\geq 0 $, that is,
\begin{align}
\Delta^{(n,\sigma)}_{2}(P)\geq P\label{cp2p-1}.
\end{align}

For $\Delta^{(n,\sigma)}_{2}(P)$, we have
\begin{align}\label{cp2p-de2p}
\Delta^{(n,\sigma)}_{2}(P)= \left(
\begin{array}{cc}
 \Delta^{(n,\sigma)}(xx^{*}) & \Delta^{(n,\sigma)}(xy^{*}) \\
 \Delta^{(n,\sigma)}(yx^{*}) & \Delta^{(n,\sigma)}(yy^{*})\\
 \end{array}
\right)=\sum_{i=1}^{n} A_{i} \otimes E_{ii},
\end{align}
where
$$A_{i}= \left(
\begin{array}{cc}
 a|x_{i}|^{2}+c_{i}|x_{\sigma(i)}|^{2} & ax_{i}\bar{y}_{i}+c_{i}x_{\sigma(i)}\bar{y}_{\sigma(i)} \\
 a\bar{x_{i}}y_{i}+c_{i}\bar{x}_{\sigma(i)}y_{\sigma(i)} &  a|y_{i}|^{2}+c_{i}|y_{\sigma(i)}|^{2}\\
 \end{array}
\right)\in M_2 $$
 and $E_{ii}\in M_n$. Since $l_{\min}(\sigma)\geq 2$, we have that $\sigma(i)\neq i$ for each $i\in \{1,2,\ldots,n\}$, which means that $\sigma$ has no fixed point. So for $i=1,2,\ldots,n$, we can choose real numbers $x_i$, $y_i$ such that each $A_i$ is invertible. For example, we can choose $x_i=\alpha i$ and $y_i=\alpha$ where $\alpha=(\frac{n(n+1)(2n+1)}{6}+n)^{-\frac{1}{2}}$, that is,
\begin{align*}
x=\left(
\begin{array}{c}
 x_1 \\
 x_2 \\
 \vdots  \\
 x_n\\
 \end{array}
\right) =\alpha\left( \begin{array}{c}
 1 \\
 2 \\
 \vdots  \\
 n\\
 \end{array}
\right),
\quad
 y=\left(
\begin{array}{c}
 y_1 \\
 y_2 \\
 \vdots  \\
 y_n\\
 \end{array}
\right)=\alpha\left(\begin{array}{c}
 1 \\
 1 \\
 \vdots\\
 1\\
 \end{array}
\right)\in \mathbb{R}^{n}.
\end{align*}
From the invertibility of each $A_{i}$, we see that $\Delta^{(n,\sigma)}_{2}(P)$ is invertible and
$$\Delta^{(n,\sigma)}_{2}(P)^{-1}=\sum_{i=1}^{n}A_{i}^{-1}\otimes E_{ii},$$
where
\begin{align*}
A_{i}^{-1}&=\left(
\begin{array}{cc}
 a|x_{i}|^{2}+c_{i}|x_{\sigma(i)}|^{2} & ax_{i}\bar{y_{i}}+c_{i}x_{\sigma(i)}\bar{y}_{\sigma(i)} \\
 a\bar{x_{i}}y_{i}+c_{i}\bar{x}_{\sigma(i)}y_{\sigma(i)} &  a|y_{i}|^{2}+c_{i}|y_{\sigma(i)}|^{2}\\
 \end{array}
\right)^{-1}\\
&=\frac{1}{ac_{i}\lambda_{i}}
\left(\begin{array}{cc}
a|y_{i}|^{2}+c_{i}|y_{\sigma(i)}|^{2}& -(ax_{i}\bar{y_{i}}+c_{i}x_{\sigma(i)}\bar{y}_{\sigma(i)}) \\
 -( a\bar{x_{i}}y_{i}+c_{i}\bar{x}_{\sigma(i)}y_{\sigma(i)})  &  a|x_{i}|^{2}+c_{i}|x_{\sigma(i)}|^{2} \\
\end{array}\right)
\end{align*}
and $\lambda_{i}=|x_{i}y_{\sigma(i)}-x_{\sigma(i)}y_{i}|^{2}$.

Note that
$$\xi=\left(\begin{array}{c}
 x \\
 y \\
\end{array}\right)=\sum_{i=1}^{n}z_{i}\otimes e_{i},$$
where $z_i=(x_i\ \ y_i)^t\in\mathbb{C}^{2}$. So we obtain
\begin{align}
\langle A_{i}^{-1}z_{i}, z_{i}\rangle &=
\frac{1}{ac_{i}\lambda_{i}} \left\langle\left(\begin{array}{cc}
a|y_{i}|^{2}+c_{i}|y_{\sigma(i)}|^{2}& -(ax_{i}\bar{y_{i}}+c_{i}x_{\sigma(i)}\bar{y}_{\sigma(i)}) \\
 -( a\bar{x_{i}}y_{i}+c_{i}\bar{x}_{\sigma(i)}y_{\sigma(i)})  &  a|x_{i}|^{2}+c_{i}|x_{\sigma(i)}|^{2} \\
\end{array}\right)\left(\begin{array}{c}
x_i\\
y_i\\
\end{array}\right),
\left(\begin{array}{c}
x_{i}\\
y_{i}\\
\end{array}\right)\right\rangle\nonumber\\
&=\frac{c_{i}\lambda_{i}}{ac_{i}\lambda_{i}}=\frac{1}{a}.\label{cp2p-1/a}
\end{align}
By Lemma \ref{le:inv}, (\ref{cp2p-1}) and (\ref{cp2p-1/a}), we have
\begin{align*}
\langle\Delta^{(n,\sigma)}_{2}(P)^{-1}\xi, \xi\rangle
&=\left\langle
\left(\sum_{i=1}^{n}A_{i}^{-1}\otimes E_{ii}\right)\left(\sum_{j=1}^{n}z_{j}\otimes e_{j}\right),
\sum_{j=1}^{n}z_{j}\otimes e_{j}\right\rangle\\
&=\sum_{i=1}^{n}\langle A_{i}^{-1}z_i, z_i\rangle=\sum_{i=1}^{n}\frac{1}{a}\\
&=\frac{n}{a}\leq 1.
\end{align*}
So $a\geq n$, and $(\ref{cp2p3})$ holds.

Assume that ($\ref{cp2p3}$) holds. Since $l_{\min}(\sigma)\geq 2$, it is not hard to see that the eigenfunction of $C_{\Theta^{(n,\sigma)}}$ is
$$g(\lambda)=\det(\lambda I_{n^2}-C_{\Theta^{(n,\sigma)}})= \lambda^{n^2-2n}(\lambda-a)^{n-1}(\lambda-a+n)\Pi_{i=1}^{n}(\lambda-c_{i}).$$
If $a\geq n$, the eigenvalues of $C_{\Theta^{(n,\sigma)}}$ are nonnegative. So $\Theta^{(n,\sigma)}$ is completely positive and (\ref{cp2p1}) holds.
\end{proof}

In Proposition 6.3 of \cite{HouLi12}, Hou, Li et al. gave similar results as Theorem \ref{th:cp2p} above. For the $D$-type linear map $\Lambda_{D}$ discussed there, all row sums and column sums of the nonnegative matrix $D$ associated to $\Lambda_{D}$ are equal to $n$. In Theorem \ref{th:cp2p} above, we have not required that.

For $\sigma=\tau_{n-1}^{n}$ ($n\geq 2$), from Lemma \ref{le:length} we see that $\tau_{n-1}^{n}$ is a cycle of length $n$, and so $l_{\min}(\sigma)=n$. Hence we obtain Theorem 2.5 of \cite{Ha03} from Theorem \ref{th:cp2p} above. If $\sigma=id_n$, then we have that $l_{\min}(id_{n})=1$. In this case, we have the following result.

\begin{proposition}\label{prop: posidn}
For any positive numbers $a,c_{1},\ldots,c_{n}$, the following conditions are equivalent:
\begin{enumerate}[label=\textup{(\roman*)}, ref=\textup{\roman*}]
\item\label{con:schur1} the matrix $$A=\left(
               \begin{array}{ccccc}
                 a+c_{1}-1 & -1 & \cdots & -1 \\
                 -1 & a+c_{2}-1 & \cdots & -1 \\
                 \vdots & \vdots & \ddots & \vdots \\
                 -1 & -1 & \cdots &  a+c_{n}-1\\
               \end{array}
             \right)$$ is positive;
\item\label{con:schur2} $\Theta^{(n,id_{n})}[a;c_{1},c_{2},\ldots,c_{n}]$ is positive;
\item\label{con:schur3} $\Theta^{(n,id_{n})}[a;c_{1},c_{2},\ldots,c_{n}]$ is completely positive.
\end{enumerate}
\end{proposition}
\begin{proof}
Suppose that $\sigma=id_{n}$ and $X\in M_{n}$. By the definition of
 $\Theta^{(n,\sigma)}$, we can see that
$$\Theta^{(n,id_{n})}(X)=A\ast X,$$
where $A\ast X$ denotes the Schur product of $A$ and $X$. Hence, using Theorem 3.7 in \cite{Paulsen}, we get the equivalence of (\ref{con:schur1}), (\ref{con:schur2}) and (\ref{con:schur3}).
\end{proof}

For general $\sigma\in S_{n}$ with $l_{\min}(\sigma)=1$, the situation becomes more complicated. In \cite{Hou11}, Qi and Hou defined a linear map $\Delta_{(t_{1},t_{2},\ldots,t_{n})}$ in some more general environment. The following result improves Proposition 2.7 in \cite{Hou11}.

\begin{corollary}\label{cor:delta}
Let $H$ and $K$ be Hilbert spaces and let $\{f_{i}\}_{i=1}^{n}$ and $\{f'_{i}\}_{i=1}^{n}$ be any orthonormal sets of $H$ and $K$, respectively. Let $F_{ji}=f'_{j}f_{i}^{*}\in B(H,K)$ be a rank one operator such that for any $x\in H$ we have $F_{ji}(x)=\langle x, f_{i}\rangle f'_{j}$, where $\langle\cdot, \cdot\rangle$ denotes the inner product on $H$. Let $\Delta_{(t_{1},t_{2},\ldots,t_{n})}:B(H)\mapsto B(K)$ be defined by
\begin{align}\label{eq:delta}
\Delta_{(t_{1},t_{2},\ldots,t_{n})}(X)=\sum_{i=1}^{n}t_{i}F_{ii}XF_{ii}^{*}-
\left(\sum_{i=1}^{n}F_{ii}\right)X\left(\sum_{i=1}^{n}F_{ii}\right)^{*}
\end{align}
for all $X\in B(H)$. Then the following conditions are equivalent:
\begin{enumerate}[label=\textup{(\roman*)}, ref=\textup{\roman*}]
\item\label{con:de1} the matrix $$A=\left(
               \begin{array}{ccccc}
                 t_{1}-1 & -1 & \cdots & -1 \\
                 -1 & t_{2}-1 & \cdots & -1 \\
                 \vdots & \vdots & \ddots & \vdots \\
                 -1 & -1 & \cdots &  t_{n}-1\\
               \end{array}
             \right)$$ is positive;
\item\label{con:de2} $\Delta_{(t_{1},t_{2},\ldots,t_{n})}$ is positive;
\item\label{con:de3} $\Delta_{(t_{1},t_{2},\ldots,t_{n})}$ is completely positive.
\end{enumerate}
\end{corollary}
\begin{proof}
Since $\Delta_{(t_{1},t_{2},\ldots,t_{n})}$ is a finite rank elementary operator \cite{Hou11}, it is not hard to see that if we let $t_{i}=a+c_{i}$ for $i=1,2,\ldots,n$ we can identify it with $\Theta^{(n,id_{n})}[a;c_{1},c_{2},\ldots,c_{n}]$. By Proposition \ref{prop: posidn}, we obtain the equivalence of (\ref{con:de1}), (\ref{con:de2}) and (\ref{con:de3}).
\end{proof}

\section{Atomicity and decomposability}\label{se:atomic}

In this section we discuss when $\Theta^{(n,\sigma)}$ is atomic and decomposable. Let $\phi:M_n\mapsto M_n$ be a linear map. In \cite{Osaka}, Osaka defined a real linear map $\tilde{\phi}:M_n(\mathbb{R})\mapsto M_n(\mathbb{R})$ by
$$\tilde{\phi}(x)=\frac{1}{2}\left(\phi(x)+\overline{\phi(x)}\right),\quad x=(x_{ij})\in M_n(\mathbb{R}),$$
where $\overline{(y_{ij})}=(\overline{y}_{ij})$ for $y=(y_{ij})\in M_n$. It is not hard to see that if $\phi$ is $k$-positive or $k$-copositive, then so is $\tilde{\phi}$ for $k= 1,2,\ldots$. The following lemma indicates that when $k=2$ and $l_{\min}(\sigma)\geq 2$ the converse is also true for $\Theta^{(n,\sigma)}$.

\begin{lemma}\label{2ptilde} Let $a,c_1,c_2,\ldots,c_n$ be positive real numbers. For $\sigma\in S_{n}$ with $l_{\min}(\sigma)\geq 2$, if $\tilde{\Theta}^{(n,\sigma)}[a;c_1,c_2,\ldots,c_n]$ is 2-positive, then $a\geq n$, and so $\Theta^{(n,\sigma)}[a;c_1,c_2,\ldots,c_n]$ is 2-positive.
\end{lemma}

\begin{proof}
It is clear that $\Theta^{(n,\sigma)}(x)=\tilde{\Theta}^{(n,\sigma)}(x)$ for $x\in M_n(\mathbb{R})$. Suppose that $\tilde{\Theta}^{(n,\sigma)}$ is $2$-positive. In the proof of Theorem \ref{th:cp2p}, for $i=1,2,\ldots,n$ we can choose real numbers $x_i$, $y_i$ such that each $A_i$ is invertible. Thus if we apply the proof of Theorem \ref{th:cp2p} to $\tilde{\Theta}^{(n,\sigma)}$, we can also get that $a\geq n$. So $\Theta^{(n,\sigma)}$ is $2$-positive by Theorem \ref{th:cp2p}.
\end{proof}

\begin{lemma}\label{2psumcop}
Suppose that $\sigma\in S_{n}$ $(n\geq 3)$ and $l_{\min}(\sigma)\geq 3$.
Let $\phi:M_n\mapsto M_n$ be a positive linear map. Suppose that $\{\phi(E_{ij})\}_{i,j=1}^{n}$ satisfy the following conditions:
\begin{enumerate}[label=\textup{(\roman*)}, ref=\textup{\roman*}]
\item\label{2psumcop1} $\phi(E_{ii})e_{j}=e_{j}^{*}\phi(E_{ii})=0$ for each $1 \leq i \leq n$ and $j\in\{1,2,\ldots,n\}\backslash \{i, \sigma^{-1}(i)\}$;
\item\label{2psumcop2}$\phi(E_{ij})=-E_{ij}$ for $1\leq i\neq j\leq n$.
\end{enumerate}
If $\phi=\varphi+\psi$, where $\varphi$ is a 2-positive linear map and $\psi$ is a 2-copositive linear map, then $\tilde{\phi}:M_n(\mathbb{R})\mapsto M_n(\mathbb{R})$ is a 2-positive linear map.
\end{lemma}

\begin{proof}
First, we show that $\psi(E_{ij})$ is a diagonal matrix for $ i\neq j$.
Since $\varphi$ is 2-positive and $\psi$ is 2-copositive, we have
\begin{align}
\left(\begin{array}{cc}
\varphi(E_{ii})&\varphi(E_{ij}) \\
\varphi(E_{ji})& \varphi(E_{jj})\\
\end{array}\right)\geq 0\quad\mbox{ and }\quad
\left(\begin{array}{cc}
\psi(E_{ii})&\psi(E_{ji}) \\
\psi(E_{ij})& \psi(E_{jj})\\
\end{array}\right)\geq 0.\label{2psumcop-pos}
\end{align}
Let $j\in\{1,2,\ldots,n\}\backslash \{i, \sigma^{-1}(i)\}$ and $1\leq i \leq n$. By condition (\ref{2psumcop1}), we have
\begin{align*}
\langle\phi(E_{ii})e_{j},e_{j}\rangle&=\langle(\varphi(E_{ii})+\psi(E_{ii}))e_{j},e_{j}\rangle\\
&=e_{j}^{*}\phi(E_{ii})e_{j}=0.
\end{align*}
Using the positivity of $\varphi$ and $\psi$, we obtain that $e_{j}^{*}\varphi(E_{ii})e_{j}=0$ and $e_{j}^{*}\psi(E_{ii})e_{j}=0$. Hence $\varphi$ and $\psi$ also satisfy condition (\ref{2psumcop1}).

Note that if $\left(
\begin{array}{cc}
x&y\\
\bar{y}&z\\
\end{array}
\right)\in M_2$ is positive and $x=0$ or $z=0$, then we must have $y=0$; any principal
submatrix of a positive matrix must be a positive matrix.
So from condition (\ref{2psumcop1}), we can see that the nonzero elements in the $n\times n$ matrices $\varphi(E_{ii})$ and $\psi(E_{ii})$ can only appear in these positions: $(i,i)$, $(i,\sigma^{-1}(i))$, $(\sigma^{-1}(i),i)$ and $(\sigma^{-1}(i),\sigma^{-1}(i))$.
From (\ref{2psumcop-pos}), for $i\neq j$ we can see that the nonzero elements of the $n\times n$ matrix $\varphi(E_{ij})$ can only appear in the positions: $(i,j)$, $(i,\sigma^{-1}(j))$, $(\sigma^{-1}(i),j)$ and $(\sigma^{-1}(i),\sigma^{-1}(j))$; the nonzero elements of the $n\times n$ matrix $\psi(E_{ij})$ can only appear in the positions: $(j,i)$, $(j,\sigma^{-1}(i))$, $(\sigma^{-1}(j),i)$ and $(\sigma^{-1}(j),\sigma^{-1}(i))$.
Hence for $1\leq i,j \leq n$ we have
\begin{align*}
\varphi(E_{ij})=&y_{ij}E_{ij}+y_{i,\sigma^{-1}(j)}E_{i,\sigma^{-1}(j)}+y_{\sigma^{-1}(i),j}E_{\sigma^{-1}(i),j}\nonumber\\
&+y_{\sigma^{-1}(i),\sigma^{-1}(j)}E_{\sigma^{-1}(i),\sigma^{-1}(j)}
\end{align*}
and
\begin{align}
\psi(E_{ij})=& z_{ji}E_{ji}+z_{\sigma^{-1}(j),i}E_{\sigma^{-1}(j),i}+z_{j,\sigma^{-1}(i)}E_{j,\sigma^{-1}(i)}\nonumber\\
&+z_{\sigma^{-1}(j),\sigma^{-1}(i)}E_{\sigma^{-1}(j),\sigma^{-1}(i)},\label{psi_ij}
\end{align}
where all $y$'s and $z$'s above are complex numbers.

For $i\neq j$, by condition (\ref{2psumcop2}) we have
\begin{align}
\phi(E_{ij})=&\varphi(E_{ij})+\psi(E_{ij})\nonumber\\
=&y_{ij}E_{ij}+y_{i,\sigma^{-1}(j)}E_{i,\sigma^{-1}(j)}
+y_{\sigma^{-1}(i),j}E_{\sigma^{-1}(i),j}
+y_{\sigma^{-1}(i),\sigma^{-1}(j)}E_{\sigma^{-1}(i),\sigma^{-1}(j)}+\nonumber\\
&z_{ji}E_{ji}+z_{\sigma^{-1}(j),i}E_{\sigma^{-1}(j),i}
+z_{j,\sigma^{-1}(i)}E_{j,\sigma^{-1}(i)}+
z_{\sigma^{-1}(j),\sigma^{-1}(i)}E_{\sigma^{-1}(j),\sigma^{-1}(i)}
\label{2psumcop-ij}\\
=&-E_{ij}\nonumber.
\end{align}
If $\psi(E_{ij})=0$, then clearly $\psi(E_{ij})$ is diagonal. Suppose that $\psi(E_{ij})\neq 0$. Since $\{E_{ij}\}_{1\leq i,j\leq n}$ are linear independent,  by comparing indices in (\ref{2psumcop-ij}) it can only happen that
\begin{enumerate}[label=\textup{(\arabic*)}, ref=\textup{\arabic*}]
\item\label{case1} $\sigma^{-1}(j)=i$ and $\sigma^{-1}(i)\neq j$;
\item\label{case2} $\sigma^{-1}(i)=j$ and $\sigma^{-1}(j)\neq i$;
\item\label{case3} $\sigma^{-1}(j)=i$ and $\sigma^{-1}(i)=j$;
\item\label{case4} $\sigma^{-1}(j)=j$ or $\sigma^{-1}(i)=i$.
\end{enumerate}
Suppose that condition (\ref{case1}) holds. From (\ref{2psumcop-ij}) it is not hard to see that $z_{\sigma^{-1}(j),i}=-y_{i,\sigma^{-1}(j)}\neq 0$ and  $z_{ji}=z_{j,\sigma^{-1}(i)}=z_{\sigma^{-1}(j),\sigma^{-1}(i)}=0$. So from (\ref{psi_ij}) we can see that $\psi(E_{ij})$ is diagonal. Similarly, $\psi(E_{ij})$ is also diagonal if condition $(\ref{case2})$ holds.

Since $l_{\min} (\sigma)\geq 3$, condition (\ref{case3}) and condition (\ref{case4}) cannot happen. If condition (\ref{case3}) holds, then $\sigma(j)=i$ and $\sigma(i)=j$. Thus there exists a cycle of length $2$ in the disjoint cycle decomposition of $\sigma$. So we have $l_{\min}(\sigma)\leq 2$ which is contradict to our assumption. Similarly, we can see that condition (\ref{case4}) cannot happen. Thus we can see that $\psi(E_{ij})$ are diagonal matrices for all $1\leq i\neq j\leq n$. Hence $\psi(E_{ij})^{t}=\psi(E_{ij})$ for all $1\leq i\neq j \leq n$.

Next, we show that $\tilde{\psi}$ is 2-positive. Since $\psi$ is positive, $\psi(x^{*})=\psi(x)^{*}$ for any $x\in M_{n}$. From discussions above, we have
\begin{align}
&\psi(E_{ij})=\psi(E_{ji}^{*})=\psi(E_{ji})^{*}=\overline{\psi(E_{ji})^{t}}=
\overline{\psi(E_{ji})}
\quad\text{for all $1\leq i\neq j\leq n$}; \label{psij}\\
&\psi(E_{ii})=\psi(E_{ii})^{*}=\overline{\psi(E_{ii})^{t}}\quad\text{for $1\leq i \leq n$}\label{psii}.
\end{align}
For each $(x_{ij})\in M_n(\mathbb{R})$, from (\ref{psij}) and (\ref{psii}) we have
\begin{align*}
\tilde{\psi}((x_{ij}))&=\frac{1}{2}\left(\psi((x_{ij})) +\overline{\psi((x_{ij}))} \right)=\frac{1}{2}\left(\sum_{i,j=1}^{n}x_{ij}\psi(E_{ij}) +\sum_{i,j=1}^{n}x_{ij}\overline{\psi(E_{ij})} \right)\\
&=\frac{1}{2}\left(\sum_{i=1}^{n}x_{ii}(\psi(E_{ii})+\psi(E_{ii})^{t})+
\sum_{1\leq i\neq j \leq n}x_{ij}(\psi(E_{ij}) + \psi(E_{ji}))\right)
\end{align*}
and
\begin{align*}
\tilde{\psi}((x_{ij})^{t})&=\frac{1}{2}\left(\psi((x_{ji})) +\overline{\psi((x_{ji}))} \right)=\frac{1}{2}\left(\sum_{i,j=1}^{n}x_{ji}\psi(E_{ij}) +\sum_{i,j=1}^{n}x_{ji}\overline{\psi(E_{ij})} \right)\\
&=\frac{1}{2}\left(\sum_{i=1}^{n}x_{ii}(\psi(E_{ii})+\psi(E_{ii})^{t})+
\sum_{1\leq i\neq j \leq n}x_{ji}(\psi(E_{ij}) + \psi(E_{ji}))\right).
\end{align*}
So we have
\begin{align}
\tilde{\psi}(X)=\tilde{\psi}(X^{t}),\label{psixt}
\end{align}
for any $X\in M_n(\mathbb{R})$.

Now for each $\left(
\begin{array}{cc}
X&Y\\
Y^{t}&Z\\
\end{array}
\right)\geq 0$ in $M_2(M_n(\mathbb{R}))$, we have
\begin{align*}
\tilde{\psi}_{2}\left(
\begin{array}{cc}
X&Y\\
Y^{t}&Z\\
\end{array}
\right)&=\left(
\begin{array}{cc}
\tilde{\psi}(X)&\tilde{\psi}(Y)\\
\tilde{\psi}(Y^{t})&\tilde{\psi}(Z)\\
\end{array}
\right)
=\left(
\begin{array}{cc}
\tilde{\psi}(X)&\tilde{\psi}(Y^{t})\\
\tilde{\psi}(Y)&\tilde{\psi}(Z)\\
\end{array}
\right)\\
&=\tilde{\psi}^{2}\left(
\begin{array}{cc}
X&Y\\
Y^{t}&Z\\
\end{array}
\right)
\geq 0,
\end{align*}
where the second equality is followed from (\ref{psixt}), and the last inequality is followed from the 2-copositivity of $\tilde{\psi}$. So $\tilde{\psi}$ is 2-positive.

Since $\tilde{\phi}=\tilde{\varphi}+\tilde{\psi}$ and both $\tilde{\varphi}$ and $\tilde{\psi}$ are 2-positive, we have that $\tilde{\phi}$ is 2-positive.
\end{proof}

Suppose that $\phi:M_{n}\mapsto M_{n}$ is a  positive linear map. $\phi$ is said to be {\it atomic} if $\phi$ can not be decomposed into a sum of a 2-positive map and a 2-copositive map. If $\phi$ can be decomposed into sums of completely positive maps and completely copositive maps, then $\phi$ is said to be {\it decomposable}, otherwise, $\phi$ is said to be {\it indecomposable}. Let $1_{k}$ denote the identity map on $M_{k}$ and $\boldsymbol{\mathrm{T}}$ denote the transpose map on $M_{n}$, the {\it partial transpose} $X^{\Gamma}$ of a matrix $X$ in $M_{k}\otimes M_{n}$ is defined by
$$X^{\Gamma}=(1_{k}\otimes\boldsymbol{\mathrm{T}})(X).$$
It is not hard to see that $\phi$ is decomposable if and only if $C_{\phi}$ can be decomposed as sums of positive matrices and matrices whose partial transpose are positive.

\begin{theorem}\label{thm:atomic}
Let $a,c_1,c_{2},\ldots,c_n$ ($n\geq 3$) be positive real numbers. Suppose that $\sigma\in S_{n}$ and $l_{\min}(\sigma)\geq 3$. If $\Theta^{(n,\sigma)}[a;c_{1},c_{2},\ldots,c_{n}]$ is positive but not completely positive, then $\Theta^{(n,\sigma)}[a;c_{1},c_{2},\ldots,c_{n}]$ is atomic. Particularly, if $n>a\geq \max\{n-1, n-(c_{1}c_{2}\cdots c_{n})^{\frac{1}{n}}\}$, then $\Theta^{(n,\sigma)}[a;c_{1},c_{2},\ldots,c_{n}]$ is atomic.
\end{theorem}

\begin{proof} Since $l_{\min}(\sigma)\geq 3$ and $\Theta^{(n,\sigma)}$ is positive but not completely positive, we have that $a<n$ by Theorem \ref{th:cp2p}.

Assume that $\Theta^{(n,\sigma)}=\varphi+\psi$, where $\varphi$ is 2-positive and $\psi$ is 2-copositive. For $1\leq i,j\leq n$, we have
\begin{equation*}
\Theta^{(n,\sigma)}(E_{ij})=\left\{\begin{aligned}
&(a-1)E_{ii}+c_{\sigma^{-1}(i)}E_{\sigma^{-1}(i),\sigma^{-1}(i)}\quad &\text{if}\ i=j\\
&-E_{ij}                                \quad &\text{if}\; i\neq j
                           \end{aligned} \right..
\end{equation*}
Hence $\{\Theta^{(n,\sigma)}(E_{ij})\}_{i,j=1}^{n}$ satisfy conditions in Lemma \ref{2psumcop}, and so $\tilde{\Theta}^{(n,\sigma)}$ is 2-positive. By Lemma \ref{2ptilde}, we have that $a\geq n$ which is a contradiction. Hence $\Theta^{(n,\sigma)}$ is atomic.

Particularly, if $n>a\geq \max\{n-1, n-(c_{1}c_{2}\cdots c_{n})^{\frac{1}{n}}\}$, from Theorem \ref{th:positive} and Theorem \ref{th:cp2p} we can see that $\Theta^{(n,\sigma)}$ is positive but not completely positive. Thus $\Theta^{(n,\sigma)}$ is atomic.
\end{proof}

\begin{remark}
For $\sigma\in S_{n}$, if $\sigma^{2}=id_{n}$, then the lengths of cycles in the disjoint cycle decomposition of $\sigma$ are not greater than 2, that is, $l_{\max}(\sigma)\leq 2$ and $l_{\min}(\sigma)\leq 2$. In Proposition 7.2 of \cite{HouLi12}, Hou, Li et al. showed that if $\sigma^{2}=id_{n}$, then $\Theta^{(n,\sigma)}[n-1;1,1,\ldots,1]$ is decomposable. In Proposition \ref{prop:decomposable} below, for $\sigma^{2}=id_{n}$ we also obtain a class of decomposable maps of the form $\Theta^{(n,\sigma)}[a;c_{1},c_{2},\ldots,c_{n}]$.  Hence for $l_{\min}(\sigma)\leq 2$, there exist positive linear maps of the form
$\Theta^{(n,\sigma)}[a;c_{1},c_{2},\ldots,c_{n}]$ which are decomposable and hence not atomic.
\end{remark}

Since $(n,n-1)=1$, from Lemma \ref{le:length} we have that $l_{\min}(\tau_{n-1}^{n})=n$. In Theorem \ref{thm:atomic}, if we let $\sigma=\tau_{n-1}^{n}$, then we obtain Theorem 3.2 in \cite{Ha03}. For $\Theta^{(n,\sigma)}[n-c;c,c,\ldots,c]$ $(c\geq0)$, the condition when it is positive and completely positive was discussed in \cite{HouLi12}. In the following corollary, we give conditions when it is atomic.

\begin{corollary}\label{cor:catomic}
Suppose that  $\sigma\in S_{n}$ and $0\leq c \leq \frac{n}{l_{\max}(\sigma)}$.  If $l_{\min}(\sigma)\geq 3$ and $0<c\leq \frac{n}{l_{\max}(\sigma)}$, then $\Theta^{(n,\sigma)}[n-c;c,c,\ldots,c]$ is atomic. If $c=0$, then $\Theta^{(n,\sigma)}[n-c;c,c,\ldots,c]$ is completely positive.
\end{corollary}

\begin{proof}
If $0<c\leq \frac{n}{l_{\max}(\sigma)}$ and $l_{\min}(\sigma)\geq 3$, then by Proposition 6.2 of \cite{HouLi12} we have that $\Theta^{(n,\sigma)}[n-c;c,c,\ldots,c]$ is positive. Since $l_{\min}(\sigma)\geq 3$ and $n-c<n$, by Theorem \ref{th:cp2p} we can see that $\Theta^{(n,\sigma)}[n-c;c,c,\ldots,c]$ is positive but not completely positive. Thus by theorem \ref{thm:atomic}, $\Theta^{(n,\sigma)}[n-c;c,c,\ldots,c]$ is atomic.

If $c=0$, then we can see that $\Theta^{(n,\sigma)}[n-c;c,c,\ldots,c]$ takes the form of (\ref{eq:delta}). By Corollary \ref{cor:delta}, it is not hard to see that $\Theta^{(n,\sigma)}[n;0,0,\ldots,0]$ is completely positive.
\end{proof}

It is clear that $\Theta^{(n,\tau_{k}^{n})}[n-1;1,1,\ldots,1]$ ($n\geq 3$, $k\in\{1,2,\ldots,n-1\}$) is the map `$\Phi^{(k)}$' defined in \cite{Hou11} if we restrict $\Phi^{(k)}$ to $M_{n}$. In \cite{Hou11}, Qi and Hou showed that if $k\neq \frac{n}{2}$, then $\Phi^{(k)}$ is indecomposable. Here we give the following result.

\begin{corollary}\label{cor:1atomic}
For each $k\in\{1,2,\ldots,n-1\}$ ($n\geq 3$), if $k\neq \frac{n}{2}$ when $n$ is even, then $\Theta^{(n,\tau_{k}^{n})}[n-1;1,1,\ldots,1]$ is atomic.
\end{corollary}

\begin{proof}
Suppose that $k\in\{1,2,\ldots,n-1\}$ ($n\geq 3$) and $k\neq \frac{n}{2}$ when $n$ is even. By Lemma \ref{le:length}, we have that $l_{\min}(\tau_{k}^{n})\geq 3$.
In Theorem \ref{thm:atomic}, if we let $a=n-1$ and $c_{1}=c_{1}=\cdots=c_{n}=1$, then we can see that $\Theta^{(n,\tau_{k}^{n})}[n-1;1,1,\ldots,1]$ is atomic.
\end{proof}

The following proposition extends Proposition 7.2 in \cite{HouLi12}.

\begin{proposition}\label{prop:decomposable}
Suppose that $\sigma\in S_{n}$ and $\sigma^{2}=id_{n}$. Let $F=\{i:\sigma(i)=i, \quad i=1,2,\ldots n \}$ and $F^{c}=\{1,2,\ldots,n\}\backslash F$. Let  $a,c_{1},c_{2},\ldots,c_{n}$ be positive real numbers. If $a\geq n-1$, $c_{i}\geq 1$ when $i\in F$ and $c_{i}c_{\sigma(i)}\geq 1$ when $i\in F^{c}$, then  $\Theta^{(n,\sigma)}[a;c_{1},c_{2},\ldots,c_{n}]$ is decomposable.
\end{proposition}
\begin{proof}

Let
$$P=\sum_{i\in F}(a+c_{i}-1)E_{ii}\otimes E_{ii}+(a-1)\sum_{i\in F^{c} }E_{ii}\otimes E_{ii}-\sum_{1\leq i\neq j\leq n\atop \sigma(i)\neq j }E_{ij}\otimes E_{ij}.$$
Note that $P$ is unitarily equivalent to $B\oplus 0$, where $B=(b_{ij})\in M_{n}$ is a Hermitian matrix satisfying: $b_{ii}=a-1$ or $a+c_{i}-1$; $b_{ij}=0$ or $-1$ (when $i\neq j$). Since $a\geq n-1$ and $c_{i}\geq 1$ when $i\in F$ , we can see that $B$ is a diagonally dominant Hermitian matrix. From the well-known strictly diagonal dominance theorem \cite{Horn}, it is not hard to see that $B$ is positive. Therefore, $P$ is positive.

Since $\sigma^{2}=id_{n}$, the lengths of cycles in the disjoint cycle decomposition of $\sigma$ are not greater than $2$. By definition we know that if $i\in F^{c}$, then $\sigma(i)\in F^{c}$, $i\neq\sigma(i)$ and $(i,\sigma(i))$ is a cycle of length $2$. So the number $k$ of elements in $F^{c}$ is even. Then $F^{c}$ consists of  $\frac{k}{2}$ pairs of elements and each pair is of the form ($i, \sigma(i)$).  For $i, \sigma(i)\in F^{c}$, without loss of generality we assume that $i<\sigma(i)$, and denote
$$Q_{i}=c_{\sigma(i)} E_{ii}\otimes E_{\sigma(i),\sigma(i)} +c_{i}E_{\sigma(i),\sigma(i)}\otimes E_{ii}
-E_{i,\sigma(i)}\otimes E_{i,\sigma(i)}-E_{\sigma(i) ,i}\otimes E_{\sigma(i), i} .$$
For $i\in F^{c}$, since $c_{i}c_{\sigma(i)}\geq 1$, it is not hard to see that the partial transpose $Q_{i}^{\Gamma}$ of $Q_{i}$ in $M_{n}\otimes M_{n}$ is positive.

For $i,j\in \{1,2,\ldots,n\}$, since $\sigma^{2}=id_{n}$, by definition we have \[\Theta^{(n,\sigma)}(E_{ii})=(a-1)E_{ii}+c_{\sigma^{-1}(i)}
E_{\sigma^{-1}(i),\sigma^{-1}(i)}=(a-1)E_{ii}+c_{\sigma(i)}
E_{\sigma(i),\sigma(i)}\]
and
\[\Theta^{(n,\sigma)}(E_{ij})=-E_{ij},\quad i\neq j.\]
It is not hard to see that the Choi matrix of $\Theta^{(n,\sigma)}$ is
\begin{align*}
C_{\Theta^{(n,\sigma)}}=&\sum_{i,j=1}^{n} E_{ij}\otimes\Theta^{(n,\sigma)}(E_{ij})\\
&=P+\sum_{i\in F^{c}\atop i<\sigma(i)}Q_{i}.
\end{align*}
So the Choi matrix of $\Theta^{(n,\sigma)}$ is the sums of positive matrices and matrices whose partial transposes are positive. From the correspondence between positive linear maps and Choi matrices discussed before, we see that $\Theta^{(n,\sigma)}$ is decomposable.
\end{proof}

\section{Separability of structural physical approximations and optimality of entanglement witnesses}\label{se:spaopt}

In this final section, as applications we give conditions to ensure the separability of the structural physical approximation of $\Theta^{(n,\sigma)}$ and the optimality of the entanglement witness associated with $\boldsymbol{\mathrm{T}}\circ\Theta^{(n,\sigma)}$.

Let $\phi$ be a nonzero positive linear map of $M_{n}$ into itself. Since $Tr(C_{\phi})=Tr(\phi(I_{n})) $, we see that $Tr(C_{\phi})>0$. Let $W=\frac{1}{Tr(C_{\phi})}C_{\phi}$. Since $W$ is a Hermitian matrix, there are $W^{+},W^{-}\geq 0$ such that $W^{+}W^{-}=0$ and $W=W^{+}-W^{-}$, and similarly for $C_{\phi}$ if we put $C_{\phi}^{+}=Tr(C_{\phi})W^{+}$ and $C_{\phi}^{-}=Tr(C_{\phi})W^{-}$. For $0\leq \lambda \leq 1$, let
$$\tilde{W}(\lambda)=\frac{1-\lambda}{n^{2}}I_{n}\otimes I_{n}+\lambda W .$$
By equation (14) of \cite{Chrus11},
$$\lambda^{*}=\frac{1}{1+n^{2}\|W^{-}\|}$$
is the maximal $\lambda$ such that $\tilde{W}(\lambda)\geq 0$.

In \cite{Stormer}, Stormer gave a formula of structural physical approximation for unital linear maps of $M_{n}$ into itself.
Generally, let $\phi$ be any nonzero positive linear map of $M_{n}$ into itself. The structural physical approximation of $\phi$ (denoted by $SPA(\phi)$) is defined as
\begin{align}
SPA(\phi)&=\tilde{W}(\lambda^{*})=\frac{1}{n^{2}}\left(1-\frac{1}{1+n^{2}\|W^{-}\|} \right)I_{n}\otimes I_{n}+\frac{1}{1+n^{2}\|W^{-}\|}W\nonumber\\
&=\frac{\|W^{-}\|}{1+n^{2}\|W^{-}\|}I_{n}\otimes I_{n}+\frac{1}{1+n^{2}\|W^{-}\|}W\nonumber\\
&=\frac{Tr(C_{\phi})^{-1}\|C_{\phi}^{-}\|}{1+ n^{2}Tr(C_{\phi})^{-1}\|C_{\phi}^{-}\|} I_{n}\otimes I_{n}+\frac{Tr(C_{\phi})^{-1}}{1+ n^{2}Tr(C_{\phi})^{-1}\|C_{\phi}^{-}\|}C_{\phi}\label{de:SPA}\\
&=\frac{1}{Tr(C_{\phi})+ n^{2}\|C_{\phi}^{-}\|}\left(\|C_{\phi}^{-}\| I_{n}\otimes I_{n}+C_{\phi} \right).\nonumber
\end{align}

Recall that a positive matrix $ A\in M_{m}\otimes M_{n}$ is said to be {\it separable} if $A=\sum_{i=1}^{k}B_{i}\otimes C_{i}$ for some $k\in \mathbb{N}$, and positive matrices $B_{i}\in M_{m}$ and $C_{i}\in M_{n}$ for $i=1,2,\ldots,k$. In the following proposition, if we let $c_{1}=c_{2}=\cdots=c_{n}=1$ and $\sigma=\tau_{k}^{n}$ ($k=1,2,\ldots,n-1$ and $k\neq \frac{n}{2}$ when $n$ is even), we obtain Proposition 4.2 in \cite{Hou12}.

\begin{proposition} Suppose that $\sigma\in S_{n}$ and $l_{\min}(\sigma)\geq 2$. For positive real numbers $a,c_{1},c_{2},\ldots,c_{n}$, if $a=n-1$ and  $\Theta^{(n,\sigma)}[a;c_{1},c_{2},\ldots,c_{n}]$ is positive, then the structural physical approximation of $\Theta^{(n,\sigma)}[a;c_{1},c_{2},\ldots,c_{n}]$ is separable.
\end{proposition}

\begin{proof}
For $\sigma\in S_{n}$, if $l_{\min}(\sigma)\geq 2$ and $a=n-1$, it is not hard to see that $C_{\Theta^{(n,\sigma)}}$ is unitarily equivalent to $G\oplus H$, where
$$G=\left(
      \begin{array}{cccc}
        n-2 & -1 & \cdots & -1 \\
        -1 & n-2 & \cdots & -1 \\
      \vdots& \vdots & \ddots & \vdots \\
       -1 & -1 & -1 & n-2 \\
      \end{array}
    \right) \in M_{n}
$$
and $H\in M_{n^{2}-n}$ is a diagonal matrix whose diagonal consists of $c_{i}$ ($i=1,2,\ldots,n$) and $0$. Since $G$ has only one negative eigenvalue: $-1$, so is $C_{\Theta^{(n,\sigma)}}$. Thus we have $\|C_{\Theta^{(n,\sigma)}}^{-}\|=1$. Since $Tr(C_{\Theta^{(n,\sigma)}})=n(n-2)+\sum_{i=1}^{n}c_{i}$, by (\ref{de:SPA}) we have
\begin{align}
SPA(C_{\Theta^{(n,\sigma)}})=&\frac{1}{Tr(C_{\Theta^{(n,\sigma)}})+ n^{2}\|C_{\Theta^{(n,\sigma)}}^{-}\|}\left(\|C_{\Theta^{(n,\sigma)}}^{-}\| I_{n}\otimes I_{n}+C_{\Theta^{(n,\sigma)}} \right)\nonumber\\
=&\frac{1}{Tr(C_{\Theta^{(n,\sigma)}})+ n^{2}\|C_{\Theta^{(n,\sigma)}}^{-}\|}\left( I_{n}\otimes I_{n}+C_{\Theta^{(n,\sigma)}} \right)\nonumber\\
=&\frac{1}{n(n-2)+\sum_{i=1}^{n}c_{i}+n^{2}}\Bigg(\sum_{i,j=1}^{n}E_{ii}\otimes E_{jj}+(n-2)\sum_{i=1}^{n}E_{ii}\otimes E_{ii}\label{SPAoftheta}\\
&+\sum_{i=1}^{n} c_{\sigma^{-1}(i)} E_{ii}\otimes E_{\sigma^{-1}(i),\sigma^{-1}(i)}-\sum_{1\leq i\neq j\leq n}E_{ij}\otimes E_{ij} \Bigg)\nonumber.
\end{align}

It is not hard to see that
\begin{align*}
&\sum_{i,j=1}^{n}E_{ii}\otimes E_{jj}+(n-2)\sum_{i=1}^{n}E_{ii}\otimes E_{ii}-\sum_{1\leq i\neq j\leq n}E_{ij}\otimes E_{ij}\\
=&\sum_{1\leq i<j\leq n}\left(E_{ii}\otimes E_{ii}+E_{jj}\otimes E_{jj}+E_{ii}\otimes E_{jj}+E_{jj}\otimes E_{ii}-E_{ij}\otimes E_{ij}-E_{ji}\otimes E_{ji}\right).
\end{align*}
Let $\sigma_{ij}=E_{ii}\otimes E_{ii}+E_{jj}\otimes E_{jj}+E_{ii}\otimes E_{jj}+E_{jj}\otimes E_{ii}-E_{ij}\otimes E_{ij}-E_{ji}\otimes E_{ji}$. To illustrate the separability of $\sigma_{ij}\in M_{n}\otimes M_{n}$, in this paragraph we let $\{e_{i}^{(2)}:i=1,2,\ldots,n\}$ and $\{e_{i}^{(n)}:i=1,2,\ldots,n\}$ denote the canonical orthonormal basis of $\mathbb{C}^{2}$ and $\mathbb{C}^{n}$, respectively. Let $\{E_{ij}^{(2)}:i,j=1,2\}$ denote the canonical matrix units of $M_2$.
Let
\begin{align*}
R
=&E_{11}^{(2)}\otimes E_{11}^{(2)}+E_{11}^{(2)}\otimes E_{22}^{(2)}+E_{22}^{(2)}\otimes E_{11}^{(2)}\\
&+E_{22}^{(2)}\otimes E_{22}^{(2)}-E_{12}^{(2)}\otimes E_{12}^{(2)}-E_{21}^{(2)}\otimes E_{21}^{(2)}\\
&=\left(
      \begin{array}{cccc}
        1 & 0 & 0 & -1 \\
        0 & 1 & 0 & 0 \\
        0 & 0 & 1 & 0 \\
        -1 & 0 & 0 & 1 \\
      \end{array}
    \right)\in M_{2}\otimes M_{2}.
\end{align*}
Let $R^{\Gamma}$ be the partial transpose of $R$ in $M_{2}\otimes M_{2}$. It is not hard to see that $R$ and $R^{\Gamma}$ are positive. From Theorem 2 of \cite{Horodecki1996} we can see that a positive matrix in $M_{2}\otimes M_{2}$ is separable if and only if its partial transpose is positive, hence $R$ is separable.
Let
$$D_{ij}=(e_{i}^{(n)}e_{1}^{(2)*}+e_{j}^{(n)}e_{2}^{(2)*})\in M_{n\times 2},$$
where $e_{1}^{(2)*}$ and $e_{2}^{(2)*}$ denote the conjugate transpose of $e_{1}^{(2)}$ and $e_{2}^{(2)}$, respectively. Note that $\sigma_{ij}=\left(D_{ij}\otimes D_{ij}\right) R \left(D_{ij}^{*}\otimes D_{ij}^{*}\right)$. Since $R$ is separable, we have that $\sigma_{ij}$ is separable.

From (\ref{SPAoftheta}) we have
$$SPA(C_{\Theta^{(n,\sigma)}})
=\frac{1}{n(n-2)+\sum_{i=1}^{n}c_{i}+n^{2}}\left(\sum_{1\leq i<j\leq n} \sigma_{ij}
+\sum_{i=1}^{n} c_{\sigma^{-1}(i)} E_{ii}\otimes E_{\sigma^{-1}(i),\sigma^{-1}(i)}\right).$$
Hence $SPA(C_{\Theta^{(n,\sigma)}})$ is separable.
\end{proof}

Let $\phi:M_{n}\mapsto M_{n}$ be a positive linear map. If $\phi$ is not completely positive, then
$$W_{\phi}=\frac{1}{n}C_{\phi}$$
is called the {\it entanglement witness} associated to $\phi$. An entanglement witness is said to be {\it optimal} if it detects a maximal set of entanglement \cite{Horodecki2000}. It was shown in \cite{Horodecki2000} that if $W_{\phi}$ has the {\it spanning property}, that is, $\mathcal{P}_{W_{\phi}}=
\{\zeta:\langle W_{\phi}\zeta, \zeta\rangle=0,\;\text{where}\; \zeta=\xi\otimes\eta\in\mathbb{C}^{n}\otimes\mathbb{C}^{n}
\}$ spans the whole space $\mathbb{C}^{n}\otimes\mathbb{C}^{n}$, then $W_{\phi}$ is an optimal entanglement witness.

For $\Theta^{(n,\sigma)}[n-c;c,c,\ldots,c]$ which was discussed in Proposition 6.2 of \cite{HouLi12}, let $c=1$ and $\sigma=\tau_{k}^{n}$ ($k=1,2,\ldots,n$). It was shown in \cite{Ha12} that if $k\neq n$ and $\frac{n}{2}$ (when $n$ ($n\geq 3$) is even), then the entanglement witness associated to $\boldsymbol{\mathrm{T}}\circ\Theta^{(n,\tau_{k}^{n})}[n-1;1,1,\ldots,1]$ is optimal. Using the method in \cite{Ha12}, in the following we give conditions when the entanglement witness associated to $\boldsymbol{\mathrm{T}}\circ\Theta^{(n,\sigma)}[n-c;c,c,\ldots,c]$ is optimal.

\begin{theorem}\label{th:optimal}
Suppose that $\sigma\in S_{n}$ ($n\geq 3$) and $l_{\min}(\sigma)\geq 3$. If
$0< c \leq \frac{n}{l_{\max}(\sigma)}$, then the entanglement witness associated to $\boldsymbol{\mathrm{T}}\circ\Theta^{(n,\sigma)}[n-c;c,c,\ldots,c]$ is optimal.
\end{theorem}
\begin{proof}
Suppose that $\sigma\in S_{n}$ and $l_{\min}(\sigma)\geq 3$. Since $0< c \leq \frac{n}{l_{\max}(\sigma)}$, by Corollary \ref{cor:catomic} we know that $\Theta^{(n,\sigma)}[n-c;c,c,\ldots,c]$ is positive. Since $\boldsymbol{\mathrm{T}}$ is positive, $\boldsymbol{\mathrm{T}}\circ\Theta^{(n,\sigma)}[n-c;c,c,\ldots,c]$ is also positive.

For $i,j\in \{1,2,\ldots,n\}$, since $\Theta^{(n,\sigma)}(E_{ii})=(n-c-1)E_{ii}+cE_{\sigma^{-1}(i),
\sigma^{-1}(i)}$ and $\Theta^{(n,\sigma)}(E_{ij})=-E_{ij}$ ($i\neq j$), we have
\begin{align*}
W_{\boldsymbol{\mathrm{T}}\circ\Theta^{(n,\sigma)}}
&=\frac{1}{n} C_{\boldsymbol{\mathrm{T}}\circ\Theta^{(n,\sigma)}}=\frac{1}{n} \left( \sum_{i,j=1}^{n} E_{ij}\otimes \boldsymbol{\mathrm{T}}\circ\Theta^{(n,\sigma)}(E_{ij})\right)\\
&=\frac{1}{n} \left( \sum_{i,j=1}^{n} E_{ij}\otimes W_{ji}^{(n,\sigma)}\right),
\end{align*}
where
\begin{equation*}
W_{ij}^{(n,\sigma)}=\left\{\begin{aligned}
&(n-c-1)E_{ii}+cE_{\sigma^{-1}(i),\sigma^{-1}(i)}\quad &\text{if}\ i=j\\
&-E_{ij}                                \quad &\text{if}\; i\neq j
                           \end{aligned} \right..
\end{equation*}

For any $n$-tuple $\theta=(\theta_{1},\theta_{2},\cdots,\theta_{n})$ of real numbers $\theta_{j}$, let
\begin{align}\label{eq: optimalS}
\mathcal{S}=\{\xi_{\theta}\otimes\xi_{\theta}:\xi_{\theta}=\sum_{j=1}^{n} e^{i\theta_{j}} e_{j}\}.
\end{align}
For each $i\in\{1,2,\ldots,n\}$, let
\begin{align}\label{eq:optimalV}
\mathcal{V}_{i}^{(n,\sigma)}=\{e_{i}\otimes e_{j}: j\neq i \;\text{and}\; j\neq \sigma^{-1}(i), \; \text{where}\; 1\leq j\leq n\}.
\end{align}

For $\eta_{\theta}=\xi_{\theta}\otimes\xi_{\theta}\in \mathcal{S}$ where $\xi_{\theta}=\sum_{j=1}^{n} e^{i\theta_{j}} e_{j}$, we have
\begin{align*}
\langle\eta_{\theta} W_{\boldsymbol{\mathrm{T}}\circ\Theta^{(n,\sigma)}},
\eta_{\theta}\rangle=&\frac{1}{n}\Bigg(\sum_{i=1}^{n}
(\xi_{\theta}\otimes\xi_{\theta})^{*}
\Big(E_{ii}\otimes((n-c-1)E_{ii}+cE_{\sigma^{-1}(i),\sigma^{-1}(i)})
\Big)(\xi_{\theta}\otimes\xi_{\theta})\\
&-\sum_{1\leq i\neq j \leq n}
(\xi_{\theta}\otimes\xi_{\theta})^{*} E_{ij}\otimes E_{ji}(\xi_{\theta}\otimes\xi_{\theta})\Bigg)\\
=&\frac{1}{n}\Bigg((n-c-1)\sum_{j=1}^{n}|e^{i\theta_{j}}|^{2}|e^{i\theta_{j}}|^{2}+
c\sum_{j=1}^{n}|e^{i\theta_{j}}|^{2}|e^{i\theta_{\sigma^{-1}(j)}}|^{2}
-\sum_{1\leq j\neq l\leq n} |e^{i\theta_{j}}|^{2}|e^{i\theta_{l}}|^{2}\Bigg)\\
=&\frac{1}{n}\left((n-c-1)n+nc-(n^{2}-n)\right)\\
=&0.
\end{align*}
Suppose that $l,m\in \{1,2,\ldots,n\}$ and
$f_{lm}=e_{l}\otimes e_{m}\in \mathcal{V}_{l}^{(n,\sigma)}$. Since $m\neq l$ and $m\neq \sigma^{-1}(l)$, we have
\begin{align*}
\langle f_{lm}W_{\boldsymbol{\mathrm{T}}\circ\Theta^{(n,\sigma)}},
f_{lm}\rangle=&\frac{1}{n}\Bigg(\sum_{i=1}^{n}(e_{l}\otimes e_{m})^{*}
\Big(E_{ii}\otimes((n-c-1)E_{ii}+cE_{\sigma^{-1}(i),\sigma^{-1}(i)})
\Big)(e_{l}\otimes e_{m})\\
&-\sum_{1\leq i\neq j \leq n}
(e_{l}\otimes e_{m})^{*} E_{ij}\otimes E_{ji}(e_{l}\otimes e_{m})\Bigg)\\
=&0.
\end{align*}

Now we will show that if $l_{\min}(\sigma)\geq 3$, the vectors defined in (\ref{eq: optimalS}) and (\ref{eq:optimalV}) span the whole space $\mathbb{C}^{n}\otimes \mathbb{C}^{n}$. For each vector $\sum_{i,j=1}^{n}x_{i}y_{j}e_{i}\otimes e_{j}\in \mathbb{C}^{n}\otimes \mathbb{C}^{n}$, it can be identified with a matrix $\sum_{i,j=1}^{n}x_{i}y_{j} E_{ij}\in M_{n}$. So we identify $\mathbb{C}^{n}\otimes \mathbb{C}^{n}$ with $M_n$. In \cite{Ha12}, Ha showed that the vectors $\xi_{\theta}\otimes\xi_{\theta}$ in (\ref{eq: optimalS})
span all symmetric matrices $E_{ii}$ and $E_{ij}+ E_{ji}$ ($1\leq i\neq j\leq n$) in $M_n$ under the identification between $\mathbb{C}^{n}\otimes \mathbb{C}^{n}$ and $M_n$.

For $1\leq i, j\leq n$, if $i\neq j$, we have either $e_{i}\otimes e_{j}\in \mathcal{V}_{i}^{(n,\sigma)}$ or $e_{j}\otimes e_{i}\in \mathcal{V}_{j}^{(n,\sigma)}$. If not, by the definition of $\{\mathcal{V}_{i}^{(n,\sigma)}\}_{i=1}^{n}$, we have that $i=\sigma^{-1}(j)$ and $j=\sigma^{-1}(i)$, that is, $\sigma(i)=j$ and $\sigma(j)=i$. So there is a cycle of length $2$ in the disjoint cycle decomposition of $\sigma$ which contradicts the assumption $l_{\min}(\sigma)\geq 3$. Thus under the identification between $\mathbb{C}^{n}\otimes \mathbb{C}^{n}$ and $M_n$, either $E_{ij}$ or $E_{ji}$ ($1\leq i\neq j\leq n$) lies in the linear span of vectors in (\ref{eq:optimalV}).

From discussion above, we can see that the vectors in (\ref{eq: optimalS}) and (\ref{eq:optimalV}) span the whole space $\mathbb{C}^{n}\otimes \mathbb{C}^{n}$. Hence $W_{\boldsymbol{\mathrm{T}}\circ\Theta^{(n,\sigma)}}$ has the spanning property, and so it is optimal.
\end{proof}

Suppose that $k\in \{1,2,\ldots,n-1\}$, $n\geq 3$ and $k\neq \frac{n}{2}$ when $n$ is even. From Lemma \ref{le:length}, we have $l_{\min}(\tau_{k}^{n})\geq 3$. In Theorem \ref{th:optimal}, if we let $\sigma=\tau_{k}^{n}$ and $c=1$, then we obtain Theorem 1 in \cite{Ha12}. In Theorem \ref{th:optimal}, if we let $c=0$, then $\Theta^{(n,\sigma)}[n-c;c,c,\ldots,c]=\Theta^{(n,\sigma)}[n;0,0,\ldots,0]
=\Delta_{(n,n,\ldots,n)}$, where $\Delta_{(n,n,\ldots,n)}$ is defined in (\ref{eq:delta}). For $\Delta_{(n,n,\ldots,n)}$, we have the following proposition.

\begin{proposition}
Suppose that $n\geq2$. Then $\boldsymbol{\mathrm{T}}\circ\Delta_{(n,n,\ldots,n)}$ is decomposable and the entanglement witness associated to $\boldsymbol{\mathrm{T}}\circ\Delta_{(n,n,\ldots,n)}$ is optimal.
\end{proposition}

\begin{proof}
From Corollary \ref{cor:catomic}, we have that $\Delta_{(n,n,\ldots,n)}$ is completely positive. So $\boldsymbol{\mathrm{T}}\circ\Delta_{(n,n,\ldots,n)}$ is decomposable.

Let $W_{\boldsymbol{\mathrm{T}}\circ\Delta_{n}}$ be the entanglement witness associated to $\boldsymbol{\mathrm{T}}\circ\Delta_{(n,n,\ldots,n)}$. For each $i\in\{1,2,\ldots,n\}$ and $n\geq 2$, let
\begin{align*}
\mathcal{V}_{i}^{(n)}=\{e_{i}\otimes e_{j}: j\neq i, \; \text{where}\; 1\leq j\leq n\}.
\end{align*}
For $n\geq 2$, just as (\ref{eq: optimalS}), let
\begin{align*}
\mathcal{S}=\{\xi_{\theta}\otimes\xi_{\theta}:\xi_{\theta}=\sum_{j=1}^{n} e^{i\theta_{j}} e_{j}\},
\end{align*}
where $\theta_{j}$ ($j=1,2,\ldots,n$) are arbitrary real numbers. Just
as the proof of Theorem \ref{th:optimal}, it is not hard to check that the vectors in $\{\cup_{i=1}^{n}\mathcal{V}_{i}^{(n)}\}\cup \mathcal{S}$ span the whole space $\mathbb{C}^{n}\otimes \mathbb{C}^{n}$ and $\langle W_{\boldsymbol{\mathrm{T}}\circ\Delta_{n}}\xi ,\xi\rangle=0$ for each $\xi\in \{\cup_{i=1}^{n}\mathcal{V}_{i}^{(n)}\}\cup \mathcal{S}$. Thus $W_{\boldsymbol{\mathrm{T}}\circ\Delta_{n}}$ has the spanning property, and so it is optimal.
\end{proof}

\bibliographystyle{amsplain}

\begin{thebibliography}{20}

\bibitem{Choi7510}
M.-D. Choi, \textit{Completely positive linear maps on complex matrices}, Linear Algebra Appl., \textbf{10}(1975),
285-290.

\bibitem{Choi7512}
M.-D. Choi, \textit{Positive semidefinite biquadratic forms}, Linear Algebra Appl., \textbf{12}(1975), 95-100.

\bibitem{Chrus07}
D. Chru\'{s}ci\'{n}ski, A. Kossakowski,
\textit{On the Structure of Entanglement Witnesses
and New Class of Positive Indecomposable Maps},
Open Sys. \& Information Dyn., \textbf{14}(2007), 275-294.

\bibitem{Chrus08}
D. Chru\'{s}ci\'{n}ski, A. Kossakowski,
\textit{A class of positive atomic maps},
J. Phys. A: Math. Theor., \textbf{41}(2008), 215201.

\bibitem{Chrus11}
D. Chru\'{s}ci\'{n}ski, J. Pytel, \textit{Optimal entanglement witnesses from generalized
reduction and Robertson maps},
J. Phys. A: Math. Theor., \textbf{44}(2011), 165304.

\bibitem{GTM163}
J. D. Dixon, B. Mortimer, \textit{Permutation Groups}, Grad. Texts in Math., vol. \textbf{163}, Springer-Verlag, New York, 1996.

\bibitem{Ha03}
K.-C. Ha, \textit{A class of atomic positive linear maps in matrix
algebras}, Linear Algebra Appl., \textbf{359}(2003), 277-290.

\bibitem{Ha12}
K.-C. Ha, \textit{Optimal witnesses detecting positive-partial-transpose entangled states in $\mathbb{C}^n \otimes \mathbb{C}^n$}, Phys. Rev. A., \textbf{86}(2012), 014304.

\bibitem{HaYu12}
K.-C. Ha, H. Yu, \textit{Optimal indecomposable witnesses without
extremality or the spanning property},
J. Phys. A: Math. Theor., \textbf{45}(2012), 395307.

\bibitem{Horn}
R. A. Horn, C. R. Johnson, \textit{Matrix Analysis}, Cambridge University Press, Cambridge, 1985.

\bibitem{Horodecki1996}
M. Horodecki, P. Horodecki, R. Horodecki, \textit{Separability of mixed states: necessary and sufficient conditions}, Phys. Lett. A., \textbf{223}(1996), 1-8.

\bibitem{HouLi12}
J. Hou, C.-K. Li, Y.-T. Poon, X. Qi, N.-S. Sze, \textit{Criteria and new classes of $k$-positive maps}, preprint, arXiv:1211.0386.

\bibitem{Kye1992}
S.-H. Kye, \textit{A class of atomic positive linear maps in $3$-dimensional matrix algebras}, in: M. Mathieu
(Ed.), Elementary Operators and Applications, World-Scientific, (1992), pp. 205-209.

\bibitem{Kye2012}
S.-H. Kye,  \textit{Facial structures for various notions of positivity and applications to the theory of entanglement}, Rev. Math. Phys., \textbf{25}(2013), 1330002.

\bibitem{Horodecki2000}
M. Lewenstein, B. Kraus, J. I. Cirac, and P. Horodecki,
\textit{Optimization of entanglement witnesses}, Phys. Rev. A., \textbf{62}(2000), 052310.

\bibitem{LW}
X. Li, W. Wu, \textit{Completely positive linear maps on maximal and minimal operator system structures}, Preprint.

\bibitem{Osaka}
H. Osaka, \textit{A series of absolutely indecomposable positive maps in matrix algebras}, Linear Algebra Appl.,
\textbf{186}(1993), 45-53.

\bibitem{Paulsen}
V. I. Paulsen, \textit{Completely bounded maps and operator algebras}, Cambridge Studies in Advanced Mathematics, \textbf{78}, Cambridge University Press, Cambridge, 2002.

\bibitem{PTT}
V. I. Paulsen, I. G. Todorov, M. Tomforde, \textit{Operator system structures on ordered spaces}, Proc. London Math. Soc. (3), \textbf{102}(2011), no. 1, 25-49.

\bibitem{Hou11}
X. Qi, J. Hou, \textit{Positive finite rank elementary operators and
characterizing entanglement of states}, J. Phys. A: Math. Theor., \textbf{44}(2011), 215305.

\bibitem{Hou12}
X. Qi, J. Hou, \textit{Characterization of optimal entanglement witnesses}, Phys. Rev. A., \textbf{85}(2012), 022334.

\bibitem{Stormer}
E. Stormer, \textit{Separable states and the SPA of a positive map}, preprint, arXiv:1206.5630.

\bibitem{Tomi}
K. Tanahashi, J. Tomiyama, \textit{Indecomposable positive maps in matrix algebras}, Canad. Math. Bull.,
\textbf{31}(1988), 308-317.

\bibitem{Woronowicz}
S. L. Woronowicz, \textit{Positive maps of low dimensional matrix algebras}, Rep. Math. Phys., \textbf{10}(1976), 165-183.

\bibitem{Chrus12}
J. P. Zwolak, D. Chru\'{s}ci\'{n}ski,
\textit{New tools for investigating positive maps in matrix algebras},
preprint, arXiv:1204.6579.

\end{thebibliography}

\end{document}